\documentclass[]{siamart}
\pdfoutput = 1
\usepackage{amsfonts}
\usepackage{color}
\usepackage{cite}

\DeclareMathOperator*{\trace}{Tr}
\DeclareMathOperator*{\argmin}{arg\,min}

\title{Sparse matrix factorizations for fast linear solvers with application to Laplacian systems\thanks{We  thank Raf Vandebril and Francois Glineur for fruitful discussions and comments.
This paper presents research results of the Belgian Network DYSCO (Dynamical Systems, Control, and Optimisation), funded by the Interuniversity Attraction Poles Programme, initiated by the Belgian State, Science Policy Office, and the ARC (Action de Recherche Concertée) on Mining and Optimization of Big Data Models funded by the Wallonia-Brussels Federation. 
    }}
\author{Michael T. Schaub~\thanks{ICTEAM, Universit\'e catholique de Louvain \& naXys, Universit\'e de Namur, Belgium (\email{mschaub@mit.edu}). Equally contributing first author; present address: Institute for Data, Systems, and Society, Massachusetts Institute of Technology.} 
\and Maguy Trefois~\thanks{ICTEAM, Universit\' e catholique de Louvain, Belgium (\email{maguy.trefois@uclouvain.be}). Equally contributing first author}
\and Paul van Dooren~\thanks{ICTEAM, Universit\'e catholique de Louvain, Belgium (\email{paul.vandooren@uclouvain.be}).}
\and Jean-Charles Delvenne~\thanks{ICTEAM and CORE, Universit\'e catholique de Louvain, Belgium (\email{jean-charles.delvenne@uclouvain.be}).}
}

\headers{SPARSE FACTORIZATIONS FOR FAST LINEAR SOLVERS}
{M.T. Schaub, M. Trefois, P. van Dooren, J.-C. Delvenne}

\begin{document}

\maketitle

\begin{abstract} 
In solving a linear system with iterative methods, one is usually confronted with the dilemma of having to choose between cheap, inefficient iterates over sparse search directions (e.g., coordinate descent), or expensive iterates in well-chosen search directions (e.g., conjugate gradients). 
In this paper, we propose to interpolate between these two extremes, and show how to perform cheap iterations along non-sparse search directions, provided that these directions can be extracted from a new kind of sparse factorization.  
For example, if the search directions are the columns of a hierarchical matrix,
then the cost of each iteration is typically logarithmic in the number of variables.  
Using some graph-theoretical results on low-stretch spanning trees, we deduce as a special case a nearly-linear time algorithm to approximate the minimal norm solution of a linear system $Bx= b$ where $B$ is the incidence matrix of a graph. 
We thereby can connect our results to recently proposed nearly-linear time solvers for Laplacian systems, which emerge here as a particular application of our sparse matrix factorization.
\end{abstract}



\section{Introduction}

Finding solutions of large linear systems of equations is a fundamental issue, underpinning most areas of mathematical sciences and quantitative research.
For instance, consider partial differential equations arising in various areas of physics, mechanics and electro-magnetics.
These have commonly to be solved numerically, and a spatial discretization of such a problem naturally leads to solving a large sparse or structured linear system~\cite{Young2014}.

In principle, two strategies to solve linear systems exist.
First, there are \textit{direct methods} \cite{Davis2006} like Cholesky factorization or Gaussian elimination. 
Those methods provide a (numerically) exact solution of the system by performing a finite number of computations. 
However, these algorithms can be computationally expensive, in particular as the full set of computations has always to be performed to obtain a problem solution, even if a coarser approximation thereof would be sufficient.

A second strategy is to use \textit{iterative methods} \cite{Elman1994,Saad2003,Young2014}, such as the Jacobi method or gradient descent.
Unlike for direct methods, the result after every step of an iterative algorithm may be interpreted as an approximate solution to the problem, which keeps getting improved until a desired stopping criterion, e.g., a predefined  precision, is reached.
As in practice the specification of the system to be solved is hardly ever exact, this ability to stop at suitable approximate solutions renders iterative methods generally less costly in terms of running time.
For instance, the complexity of direct Gaussian elimination for a system of size $n$ is $\mathcal{O}(n^3)$.
In contrast, the iterative Jacobi method takes only $\mathcal{O}(Nn^2)$ time.
Here, $N$ is the number of iterations needed, which can usually be kept small.

However, when the system size $n$ is very large, effectively all classical direct and iterative methods become computationally prohibitive, unless the matrix is known to have a special structure (banded, Toeplitz, semiseparable, etc.). 
Methods which provide faster means for solving linear systems are thus highly demanded. 

\subsection{Background and Related work}
The success of any iterative update scheme in solving a linear system depends on two intertwined factors.
On the one hand, we would like to design our iterations such that each update brings us as close as possible to the true solution.
On the other hand, we would like to make each iteration computationally as cheap as possible.

Let us initially consider the first of these two objectives here.
Trivially, the update that would bring us closest to the true solution entails finding the correct solution directly, and thus requires only one iteration. 
However, this is clearly not feasible, if our initial problem evaded direct solution methods.
A more realistic scheme, aiming to bring us as close as possible to the desired solution would be conjugate gradient descent, which tries to find good search directions at each step using gradient information.
The downside of an approach like gradient descent is that each step can be computationally very costly, e.g., as in general all coordinates have to be updated at each step.

This bring us back to the second objective mentioned above: making each iteration as computationally cheap as possible.
On this end of the methodological spectrum there are approaches like (canonical) coordinate descent.
Here the idea is to keep the updates very sparse and only update one (or a small number of $k$) coordinates at a time, thereby facilitating cheap iterations.
However, as this imposes quite strong restrictions on the allowed search directions, this results in general in a large number of iterations needed, possibly outweighing the gain in computational complexity for each iteration.

Recently, Spielman and Teng~\cite{Spielman2004} provided a seminal contribution and showed that one can construct iterative algorithms to solve symmetric, diagonally dominant (SDD) systems in nearly-linear running time. 
Here, \textit{nearly-linear} refers to a complexity of the form $\mathcal{O}(\ell\log^c\ell \log(\varepsilon^{-1}))$, where $\ell$ is the number of nonzero entries in the system matrix, $c$ is an arbitrary positive constant, and $\varepsilon$ is a desired accuracy to be reached.
These results have been further improved and simplified in the last decade \cite{Cohen2014,Kelner2013,Koutis2010,Koutis2011,Koutis2012,Lee2013}, and there is now a substantial literature on solving SDD systems effectively in nearly-linear time. 
Interestingly, all these algorithms follow essentially the same paradigm. 
The problem is first reduced to solving a system of the form $Lx = b$, where $L$ is the Laplacian matrix of an undirected graph. 
The Laplacian system is then solved efficiently using graph theoretic techniques.

\subsection{Main contributions}
We provide a sparse matrix factorization that enables the construction of fast iterative algorithms.
Namely, using our $k$-sparse matrix factorization allows for cheap iterative updates in efficient directions.

The key question we address is in how far cheap, coordinate descent like updates can also be performed in more flexible search directions. 
As we show in the following the answer is indeed affirmative.
If the iterative updates are performed along directions $q_i$ that can be assembled into a $k$-sparse decomposable matrix $Q=[q_1,\ldots,q_n]$, then we can always perform cheap iterative updates, despite the fact that the search direction may not have sparse support, i.e., $Q$ might be a dense matrix.
This significantly enlarges the array of possible search directions and paves the way for efficient algorithms that can benefit from both cheap updates and well-chosen search directions.

Remarkably our $k$-sparse factorization is applicable for a variety of matrices with seemingly disparate structures.
In particular, we can design iterative algorithms for sparse, hierarchical, semiseparable, or Laplacian matrices, with a complexity similar to specially tailored algorithm for those respective classes. 
In the case of Laplacian systems (and therefore all SDD systems through the usual reduction), our approach differs from previous work in that we take a different, matrix-theoretic approach, rather than relying purely on graph-theoretic machinery to achieve a nearly-linear complexity.
Finally, we show that this algorithm can be applied to solve Laplacian systems in nearly linear time, thereby establishing a connection to the previous literature.
Rather than emphasizing one particular application and providing detailed simulations for our algorithms, the focus of the present paper is on the theoretical development of a new sparse matrix factorization and its algebraic properties, which may then be used in different contexts.

Note that both sparse and dense systems are in principle amenable for a $k$-sparse decomposition.
Therefore, in principle, the target systems for our $k$-sparse matrix factorization and the associated iterative solution strategy may be dense or sparse.
For instance, Laplacian systems, which serve as our final application example in this paper, are typically sparse systems.
Nevertheless, the theory developed is equally applicable to dense systems as will become apparent when discussing hierarchical matrices.
Of course, in the case of very large dense systems, one may have to find efficient representations or approximations for storing such data (e.g., using hierarchical matrices~\cite{Hackbusch1999,Hackbusch2015}, or semiseparable matrices~\cite{Vandebril2007,Vandebril2008}).
This is a challenge in its own right, not addressed in the present manuscript.

\subsection{Outline of the paper}
In Section 2, we first review some preliminaries for iteratively solving linear systems and set up some notation
In Section 3, we then motivate and define our $k$-sparse matrix factorization.
We highlight some properties of this factorization and show how it enables an iteration of the form (\ref{eq0}) to be computed in $\mathcal{O}(k)$ time.
We then discuss, how these cheap iterations can be utilized to construct fast iterative solvers for linear systems.
In Section 4, we review several examples of $k$-sparsely factorizable matrices, including some sparse matrices, hierarchical matrices, semi-separable matrices, as well as the incidence matrices of trees.
Of particular interest here are hierarchical matrices \cite{Hackbusch1999,Boerm2003,Grasedyck2003}, which are an example of $k$-sparse factorizable matrices for which $k$ does not depend on the size of the matrix. 
In Section 5, we present fast iterative solvers for systems of hierarchical matrices, based on $k$-sparse decompositions.
In Section 6 we then show how similar techniques can be applied if the system matrix is the incidence matrix of a graph, and how this naturally leads to an algorithm for solving a Laplacian system in nearly-linear time.
Section 7 concludes the paper and discusses possible avenues for future work.
To improve readability, some technical proofs are reported in the appendix.

\section{Preliminaries}\label{sec:prelim}
For simplicity of notation we will consider only real vectors and matrices, although generalizations to the complex case are straightforward.
In the sequel, the index variable $t$ will be reserved to denote the $t$-th iterate of a vector ($x$, or $y$ respectively).
Otherwise, an indexed vector $v_i$ is to be interpreted as the $i$th column vector of a set of column vectors (usually associated with a corresponding matrix $V=[v_1,v_2,\ldots]$).

From an abstract point of view, we consider the problem of finding the minimal norm vector $x$ within an affine space $\mathcal X$.
Let $v \in \mathcal X$ be any point in our affine space.
Then by updating $x$ within this search space along a set $\{q_i\}$ of chosen search directions spanning $\mathcal X  - v$, one can find the minimal norm solution of $x$. 
More precisely, starting from an $x_0 \in \mathcal X$ we iteratively solve:
\begin{flalign}\label{eq:affine_search}
    \min \;&\|x\| \\
    \text{s.t. }& x- v \in \text{span}(\{q_i\}).\nonumber
\end{flalign}
As we review in next section, this problem is closely connected to iteratively solving a linear system, and the natural updates are of the form:
\begin{equation}\label{eq0}
x_{t+1} = x_t - \frac{x_t^Tq_i}{q_i^Tq_i}  q_i
\end{equation}

The goal of this work is to show that if the search directions for problem \eqref{eq:affine_search} are such that they correspond to the columns $q_i$ of a matrix $Q$ that is $k$-sparsely factorizable, then all iterative updates of the form \eqref{eq0} can be performed in $\mathcal{O}(k)$ time.
Here $k$ is usually much smaller than the dimension of the search space, thereby facilitating fast iterative updates schemes, as we will see in the subsequent sections.

\subsection{Underdetermined systems}

Given a compatible linear system $Ax = b$, we are looking for the optimal solution of the following optimization problem:
\begin{flalign}
    \min \;&\|x\| \\
    \text{s.t. }& Ax=b, \nonumber
\end{flalign}
where $\|x\| := \sqrt{x^Tx}$.
We denote this optimal solution by $x^*$: 
\begin{equation}
\label{eq1}
x^* := \arg\min_{s.t. Ax = b} \|x\|,
\end{equation} 

This problem can be readily solved as follows. 
Suppose we are given a matrix $Q$ where the columns $q_i$ form a basis of the null space, $\text{null}(A)$, of $A$. 
If $x_0$ is a feasible solution to $Ax = b$, we can write (\ref{eq1}) as  
\begin{equation}\label{eq:aff}
x^* := \arg\min_{s.t.\  x = x_0 + Q  y} \|x\|,
\end{equation}
for some unknown vector $y$.
Consequently, we may compute increasingly accurate approximations of $x^*$ by iteratively updating $x$ according to: 
\begin{equation}
    x_{t+1} = x_t + \alpha_t^* q_i \qquad \text{with} \qquad \alpha_t^*  =  \arg\min_{\alpha_t \in \mathbb{R}} \|x_t + \alpha_t q_i\| = - \frac{ x_t^T q_i}{q_i^T q_i}.
\end{equation}
Thus each iteration is of the form (\ref{eq0}).
We remark that these updates may be interpreted in the context of a (randomized) Kacmarz scheme as discussed in the Appendix.
If we start with a feasible solution $x_0$, each iterate $x_t$ is an exact solution of $Ax=b$, since all updates added to $x_0$ are in the null space of $A$.
Therefore, the above iterative method converges to the optimal $x^*$.

\subsection{Overdetermined and square systems}
\label{overdetcase} Iteration (\ref{eq0}) also appears naturally when iteratively solving an overdetermined system:
\begin{equation}
\argmin_{y} \|Ay-b\|. 
\end{equation} 
By simply making the substitution $x=Ay-b$, we can transform the above into the equivalent problem: 
\begin{flalign}\label{eq:overdetermined}
    \min \;&\|x\| \\
    \text{s.t. }& x+b \in \text{Im}(A),\nonumber
\end{flalign}
i.e., we are again trying to find the minimum norm solution of $x$ within an affine space.
Now an arbitrary $y_0$ will provide a starting point $x_0 = Ay_0-b$ for an iterative update procedure, and the search directions can be set to $Q=A$. 
Let $e_i$ denote the $i$-th unit coordinate vector.
It is now easy to see that our update rule \eqref{eq0} for $x$ amounts to dual updates in $y$ in coordinate descent form:
$$y_{t+1} = y_t - \dfrac{(Ay_t -b)^Tq_i}{\|q_i\|^2}e_i = y_t + \alpha_t^* e_i$$

Hence, we can iteratively construct the solutions in $y$ and $x$ by keeping track of the stepsizes $\alpha^*_t$ in the directions of $Q$. 
One may of course alternatively choose $Q=AS$, for any full-row-rank matrix $S$.
The case of a square invertible system corresponds to the overdetermined scenario in which the minimum-norm solution $x$ is zero.
Most of our results for the underdetermined case can thus be simply recast, \textit{mutatis mutandis}, to the overdetermined or square invertible setting, and vice versa.

\section{A new sparse matrix factorization for fast iterative updates}

\subsection{A $k$-sparse matrix factorization enabling efficient updates for iterative algorithms}
We are now prepared to introduce the notion of $k$-sparse matrix factorization. 
Our motivation for this factorization is that it should enable fast iterative updates of the form \eqref{eq0}, i.e., we want to compute \textit{any} iteration
$$x_{t+1} = x_t - \frac{x_t^Tq_i}{q_i^Tq_i}  q_i,$$
in $\mathcal{O}(k)$ time, if $q_i$ is a column of the $k$-sparsely factorizable matrix $Q= [q_1, q_2, \ldots]$.

The underlying idea here is akin to the case where $q_i$ is a sparse vector with only $k$ non-zero entries.
Then just $k$ non-zero products need to be computed. 
Hence, the computational cost of the update is $\mathcal{O}(k)$. 
However, in order to solve a generic linear system efficiently, we need to ensure that we can find a set of vectors $\{q_1,\ldots,q_n\}$ such that \textit{all} necessary iterative updates can be performed with this complexity.
This will be the key ingredient of our results on linear solvers presented in Section \ref{sec:iterative_solvers}.

\begin{definition}[Support and sparsity of vectors and matrices]
    The support of a vector ${v = \left(v^1, \ldots, v^m\right)^T \in \mathbb{R}^m}$ is the set of indices of the nonzero entries of $v$: 
    $$\mathrm{supp}(v) = \{i \in \{1,\ldots, m\} : v^i \neq 0\}.$$ 
    A vector $v \in \mathbb{R}^m$ is said to be $k$-sparse, if the size of its support, $|supp(v)|$, is less than or equal to $k$. Similarly, a matrix is said to be \textit{$k$-column} (\textit{$k$-row}) \textit{sparse} if each of its columns (rows) is $k$-sparse.
\end{definition}

Suppose that $x_t$ is not stored in the canonical basis, but in a different set of coordinates encoded by a matrix $C$. 
That is, instead of performing iterations \eqref{eq0} on $x_t$, we keep track of a vector $y_t$ such that $x_t=Cy_t$.
To yield a sparse update, we may choose $C$ such that $q_i$ is sparse in this representation, i.e., $q_i=C d_i$, where $d_i$ is a $l$-sparse vector.
This leads to an iteration of the form:
\begin{equation*}
    Cy_{t+1} = Cy_t - \dfrac{x_t^Tq_i}{q_i^Tq_i}Cd_i
\end{equation*}
Using this representation, every update would be sparse in that it would only effect $l$ components of $y$.
However, this is not enough to perform each iteration (\ref{eq0}) fast, as one also needs to compute the scalar product $x_t^Tq_i$, which in the new basis becomes $y_t^T C^TCd_i$, i.e., the iteration in terms of $y_t$ is of the form:
\begin{equation*}
y_{t+1} = y_t - \dfrac{y_t^TC^TCd_i}{d_i^TC^TCd_i}d_i
\end{equation*}
To bound the complexity of this operation, one must understand the sparsity pattern of $C^TC$, which is dictated by how the supports of the columns of $C$ overlap. 
Observe that the entry $[C^TC]_{ij}$ contains the scalar products between the $i^{th}$ and the $j^{th}$ column of $C$.
Whence, if every column of $C$ overlaps in support with at most $c$ other columns, then every column of $C^TC$ contains at most $c$ non-zero entries. 
If we can find a matrix for which this is true, then $C^TCd_i$ is a $k=cl$ sparse vector, since $d_i$ is $l$-sparse, and $y_t C^TCd_i$ is computed in time $\mathcal{O}(k)$. 
If we compile all such vectors $q_i$ into a matrix $Q$, then we say that $Q=CD$ is a $k$-sparse factorization.

While this reasoning provides us with some intuition, this definition must in fact be improved to reach tighter complexity bounds.
First, we can exploit the symmetry of $C^TC$, by noting that it can be decomposed as $C^TC= U^T+U$, where $U$ is an upper-triangular matrix.
Observe that the number of non-zero entries in the $i$th column of $U^T$ (or $i$th row of $U$) is bounded by the number of columns $c_j$ that overlap with $c_i$ for $j \geq i$. 
Second, two columns of $U^T$ may have their non-zero entries at the same positions. 
Therefore, the support of the sum of two columns does not necessarily increase.
To bound the complexity we need to look at the size of the union of supports of all columns $u_j$ of $U^T$, for which $j$ belongs to the support of $d_i$.
This number can indeed be much lower than the approximate estimate $cl$ above.
This justifies the following definition.

\begin{definition}
\label{defi22}
    Suppose a matrix $Q \in \mathbb{R}^{m \times n}$ has a factorization $Q=CD$. Let us denote the columns of $C\in \mathbb R^{m\times p}$ and $D \in \mathbb R^{p\times n}$ by $c_i$ and $d_j$, respectively. 
    We define the forward-overlap $FO(c_i)$ of a column $c_i$ to be the list of columns $c_j$, with $j \geq i$, that have a support overlapping with the support of $c_i$.
    We call the factorization $Q=CD$  \emph{$k$-sparse} if  $\left|\cup_{i \in \mathrm{supp}(d_j)} FO(c_i)\right| \leq k$ for all $j$ (see Figure \ref{fig2} for an illustration).
    Without loss of generality each column of $C$ and each row of $D$ is supposed to be nonzero.
\end{definition}

\begin{figure}[tb!]
\center
\includegraphics[width=\textwidth]{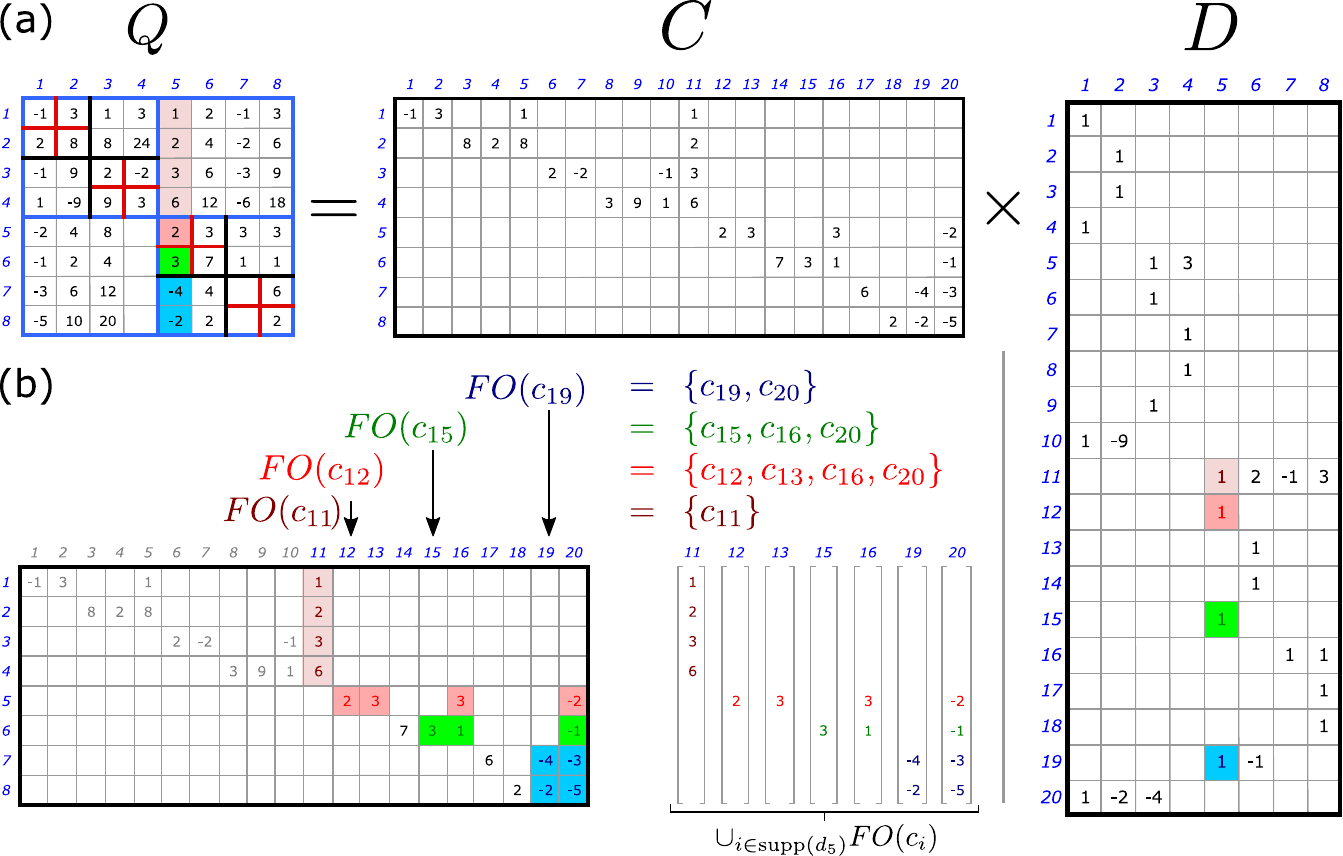}
\caption{(a) An example of $8$-sparse factorization. (b) Illustration of the forward overlap of all columns $i \in \text{supp}(d_5)$}
\label{fig2}
\end{figure}
The example in Figure \ref{fig2} shows an 8-sparse factorization of the given matrix $Q$. 
For instance, one can easily check that the forward overlap of column $c_{12}$ is $FO(c_{12}) = \{c_{12}, c_{13}, c_{16}, c_{20}\}$, and e.g.
$$\left|\cup_{i \in supp(d_5)}FO(c_i)\right| = \left|\{c_{11}, c_{12}, c_{13}, c_{15}, c_{16}, c_{19}, c_{20}\}\right| = 7\leq k = 8.$$

To gain some further intuition, let us consider an alternative definition of a sparse factorization. 
We define a partial order on the columns of $C$ with the following properties. 
First, only columns $c_i$ with overlapping support are comparable. 
Second, every subset $T_i= \{c_i,\ldots\}$ spanning a column $q_i$ has an upper set of at most $k$ elements.
The upper set is here defined as the union of $T_i$ and all columns of $C$ larger than any element of $T_i$ in the partial order.
Indeed the factorization $Q=CD$ expresses nothing but the fact that every column $q_i$ is a linear combination of a set of columns of $C$ with coefficients given by entries of $i$th column of $D$.

The following properties of a $k$-sparse factorization are worth noting.
\begin{enumerate}
\item Any $m$-by-$n$ matrix $Q$ is $\min(m,n)$-sparsely factorizable with either $Q=QI$ or $Q=IQ$. 
    Similarly, it is easy to see from an SVD that every rank $k$ matrix is $k$-sparse factorizable.
\item \label{proper2}  If $Q = CD$ is a $k$-sparse factorization, then for every column $c_i$ of $C$, $|FO(c_i)|\leq k$, $C$ is $k$-row sparse and each column of $D$ is $k$-sparse.
\item Conversely, a matrix $C$ such that $|FO(c_i)| \leq k$ for all columns $c_i$ is trivially $k$-sparsely factorizable. A $k$-column sparse matrix $D$ is also trivially $k$-sparsely factorizable.
\item If $Q = CD$ is a $k$-sparse factorization and $F$ is $f$-column sparse, then $QF = C(DF)$ is a $kf$-sparse factorization of $QF$.
\item If $Q_1 = C_1D_1$ is a $k_1$-sparse factorization and $Q_2 = C_2D_2$ is a $k_2$-sparse factorization, then the matrix $(Q_1^T\: Q_2^T)^T$ is $(k_1 + k_2)$-sparsely factorizable.
In order to see this, we write 
$$\begin{pmatrix} Q_1 \\ Q_2\end{pmatrix} = \begin{pmatrix} C_1 & \\ & C_2\end{pmatrix}\begin{pmatrix}D_1 \\ D_2\end{pmatrix}.$$ 
In particular, if $Q_2$ is the identity, the compound matrix is $(k_1+1)$-sparsely factorizable.
\end{enumerate}

The following theorem establishes the running time of $N$ iterations of the form (\ref{eq0}), when the vectors $q_i$ are the columns of a $k$-sparsely factorizable matrix. 
The proof of the theorem is given in the appendix.
\begin{theorem}
\label{thm1}
Let $Q \in \mathbb{R}^{m \times n}, C \in \mathbb{R}^{m \times p}$ and $D \in \mathbb{R}^{p \times n}$ be matrices such that $Q = CD$ is a $k$-sparse factorization of $Q$, and consider iterations of the form (\ref{eq0}) that start from an arbitrary vector $x_0 \in \mathbb{R}^m$. 
If every $q_i$ in \eqref{eq0} is a column of $Q$, then the computational complexity of running $N$ iterations of \eqref{eq0} is: 
$$\mathcal{O}(Nk + (m+n)k^2).$$
With the same complexity, we can compute a $y_N$ such that $x_N = x_0 + Qy_N$, where $x_N$ denotes the vector resulting from the $N$ first iterations. 
By applying sufficiently many iterations of form \eqref{eq0} we thus obtain both the solution to the primal problem in $x$, as well as the solution to the dual problem in $y$.
\end{theorem}

The remarkable point about Theorem \ref{thm1} is that the running time of \emph{each iteration} is merely $\mathcal{O}(k)$, even if some columns of $Q$ are full.
Hence, if $k \ll m$, then the cost per iteration can be largely reduced through the use of a $k$-sparse factorization, and the overhead term $(m+n)k^2$ is more than compensated.

\subsection{Ensuring fast convergence by randomized updates}\label{sec:convergence}

From our discussion above, we know that after sufficiently many iterations \eqref{eq0} over all columns of $Q$, $x_t$ converges to:
\begin{equation}\label{minnorm}
x^* = \argmin_{x \in x_0 + \textrm{Im } Q} \|x\|_2 
\end{equation}
However, to ensure that we can construct an efficient algorithm based on such cheap updates, we need to guarantee that the required number of updates is not too large, as this would undermine the purpose of the fast updates.
Stated differently, we need the convergence rate of our iterations to be not too slow.

Remarkably, one can indeed ensure a sufficient convergence rate using a random sampling of the columns of $Q$.
To this end, at each iteration randomly select a column $q_i$ with probability proportional to $\|q_i\|$. 
This guarantees a convergence rate of the form
$$
\mathbb{E} \|x_{t} - x^*\|^2_2 = \left( 1-\dfrac{\sigma^2_{\min}(Q)}{\|Q\|^2_\textrm{Frob}} \right)^t \|x_0 - x^*\|^2_2, 
$$
where $\|Q\|_\textrm{Frob}= \sqrt{\trace Q^TQ}$  is the Frobenius norm and $\sigma^2_{\min}(Q)=\lambda_{\min}(Q^TQ)$ is the smallest nonzero squared singular value~\cite{Strohmer2009,Gower2015}.
The proof of this result is provided in the appendix.
There we also discuss interpretations of the here presented scheme in terms of a randomized Kacmarz or randomized coordinate descent method -- with a particular choice of update directions.

The above results states that the expected error in computing $x^*$ is decreased by an order of magnitude, e.g., by a factor of $\delta^{-1} = 10$ after a number of iterations given by 
\begin{equation}\label{eq:Ncondnum}
N_1 = \dfrac{-\log(\delta^{-1})}{\log(1-\sigma^2_{\min}(Q) / \|Q\|^2_\textrm{Frob})} \approx \mathcal{O}(\|Q\|^2_\textrm{Frob}/\sigma^2_{\min}(Q))\end{equation}
The main challenge for the construction of a fast algorithm is thus to find a matrix $Q$ spanning the desired search space, with efficient $k$-sparse factorization and low `condition number' $\|Q\|_\textrm{Frob}/\sigma_{\min}(Q)$.
Note that scaling each column of $Q$ by a different scalar will not change whether or not the updates will converge.
Neither, will it change the complexity of each update (as columns of $Q$ only matter for their directions).
However, scaling the column may change the  `condition number' of $Q$, and hence the bound on the convergence time.

\subsubsection{The underdetermined case} \label{sec:underdetermined_convergence}
Let us develop the above reasoning somewhat further for the underdetermined case. 
One seeks the minimum-norm solution $x^*$ to $Ax=b$, where $A$ is an $n$-by-$m$ matrix with full-row-rank.  
Therefore it can be decomposed as $A =\begin{pmatrix} E & F \end{pmatrix}$, where $E$ is an invertible $n \times n$ submatrix of $A$.

A matrix $Q$ whose columns span the null space of $A$ can then be constructed as:
\begin{equation}
\label{eq10}
Q = \begin{pmatrix} E^{-1}F \\ -I_{m-n} \end{pmatrix},
\end{equation}
where $I_{m-n}$ is the identity matrix of dimension $m-n$. 
We clearly have $AQ = 0$, and thus the columns of $Q$ belong to the null space of $A$.
The rank of $Q$ is $m-n$, which is the dimension of $\text{null}(A)$.

Moreover, we have that $\sigma_{\min}^2(Q)=\lambda_{\min}(  F^TE^{-T}E^{-1}F + I_{m-n} ) \geq 1$. 
The number of steps to decrease the error by one order of magnitude is therefore at most of the order of:
\begin{equation}
N_1 = \mathcal O \left (\dfrac{\|Q\|_\textrm{Frob}^2}{\sigma_{\min}^2(Q)} \right) = \mathcal O(\|E^{-1}F\|_\textrm{Frob}^2+m)
\label{eq:Ncondnum2}
\end{equation}

Note that from the elementary properties of sparse factorization that if $E^{-1}=CD$ is $k_0$-sparsely factorizable, $F$ is $f$-column sparse, then $E^{-1}F$ is $k_0f$-sparsely-factorizable and $Q$ is $k=(kf+1)$-sparsely-factorizable: 
\begin{equation}
\label{eq11}
Q = \tilde{C}\tilde{D} = \begin{pmatrix} C & 0 \\ 0 & I_{m-n}\end{pmatrix}\begin{pmatrix} DF \\ -I_{m-n}\end{pmatrix}.
\end{equation}

Hence, we have a good complexity if we can find an invertible square submatrix $E$ such that  $\|E^{-1}F\|_\textrm{Frob}$ is small, and the resulting $Q$ is $k$-sparsely factorizable, for low $k$. 

We still have to find a fairly good initial guess, however.
A simple initial solution is given by $x_0=(\begin{smallmatrix} E^{-1}b \\ 0 \end{smallmatrix})$, which can be shown to fulfill the following error bound:
\begin{flalign*} 
    \|x_0\|^2 &=\|E^{-1}b\|^2 =\|E^{-1}Ax^*\|^2  =\left \| \begin{pmatrix} I & E^{-1}F \end{pmatrix} x^* \right \|^2   \\
        &\leq \left \|\begin{pmatrix} I & E^{-1}F \end{pmatrix}\right \|_\text{Frob}^2 \|x^*\|^2 = \mathcal{O}(n + \|E^{-1}F\|^2_\text{Frob}) \|x^*\|^2.  
\end{flalign*} 
Overall, reducing the initial relative error 
$$\epsilon_0 =\|x_0-x^*\|/\|x^*\| \leq 1 + \|x_0\| / \|x^*\| =\mathcal{O}\left(\sqrt{n + \|E^{-1}F\|^2_\text{Frob}}\right)$$ 
to a prescribed value $\epsilon$, requires thus a reduction by $\mathcal O (\log (n + \|E^{-1}F\|^2_\text{Frob})/2 + \log \epsilon^{-1})$ orders of magnitude, which is also in $\mathcal O (\log (m + \|E^{-1}F\|^2_\text{Frob}) + \log \epsilon^{-1})$ given that $n \leq m$. 

In summary, denoting $\kappa= m+\|E^{-1}F\|^2_\text{Frob}$, we find that it takes $N_1 = \mathcal O (\kappa)$ iterations to decrease the error by an order of magnitude.
Further, it takes  $\mathcal O (\log (\kappa \epsilon^{-1}))$ orders of magnitude to achieve relative accuracy $\epsilon^{}$. 
Following Theorem~\ref{thm1}, the total complexity is thus $\mathcal O (\kappa \log (\kappa \epsilon^{-1}) k + m k^2)$. 

\section{Classes of sparsely factorizable matrices}
Many modern and classical methods aim at exploiting particular structure in the system matrix for fast algorithms.
Table \ref{tab:comparision} provides an overview of results known from the literature and the $k$-sparse factorization approach presented in this paper.
Interestingly, our $k$-sparse matrix factorization approach provides good complexity results for a range of different matrix types, and might thus be seen as a general 
framework for seemingly different matrix structures.
We will now discuss some classes in more detail.

Let us start with some intuitive examples first.
A simple case is the overdetermined  system $Ay=b$ where $A$ is $k$-column-sparse. 
In this case, taking $Q=A= I A$ as a trivial $k$-sparse factorization, and our algorithm can be seen as a randomized Kacmarz scheme for the normal equation $A^Tx = A^Tb$, which keeps track of the updates in the $x$ coordinates but also in the $y$ coordinates. 
In the space of $y$, this is simply coordinate descent with a cost $\mathcal O(k)$, as discussed in Section \ref{overdetcase}.
The total cost amounts to $\mathcal O (Nk)$ as the overhead cost becomes irrelevant when $C$ in the $Q=CD$ decomposition is the identity.

If $A$ is $k$-row-sparse and invertible then $Q=I A^T$ is a $k$-sparse factorization.
In this case a trivial modification of the algorithm in the proof of Theorem \ref{thm1} simply coincides again with a randomized Kacmarz scheme~\cite{Strohmer2009} (see Appendix).

\begin{table}[bt!]
    \centering
    \caption{Complexity of solving (compatible) structured linear systems with a $k$-sparse matrix factorization approach compared to known results in the literature.} 
    \label{tab:comparision}
    \begin{tabular}{c|c|c}
        Structure & k-sparse factorization & Literature \\
        \hline 
        $k$ row/column sparse & $\mathcal{O}( N k)$& $\mathcal{O}( N k)$ (randomized Kacmarz \cite{Strohmer2009}) \\
        Hierarchical & $\mathcal{O}( N \log(n)+ n \log^2(n))$& $\mathcal{O}( n \log^2(n))$ (direct method \cite{ambikasaran})\\
        semiseparable & $\mathcal{O}( N \log(n)+ n \log^2(n))$& $\mathcal{O}(n)$ \cite{Vandebril2007,Vandebril2008}\\
        Laplacian & $\mathcal O (m \log^2 n \log\log n \log(m \epsilon^{-1}))$  & \cite{Kelner2013} (similar to this paper) \\
                  & (Thm. \ref{thmLaplacian})& \cite{Cohen2014} (fastest algorithm) \\
    \end{tabular}
\end{table}

\subsection{Hierarchical matrices}

In the following, we will discuss hierarchical $\mathcal{H}_r$-matrices~\cite{Hackbusch2015}, originally introduced by Hackbusch~\cite{Hackbusch1999}, and show that they are $k$-sparsely factorizable.
Importantly, in this case $k$ depends only on the height and the degree of the hierarchical structure.

\subsubsection{Definition of an $\mathcal{H}_r$-matrix}\label{sec:31}

As the name suggests, $\mathcal{H}_r$-matrices are intimately related to hierarchical structures. 
As a hierarchy may be aptly represented as a tree we introduce these matrices here with the help of (tree-)graphs.
As we will see this also enables us to establish a connection to graph-theoretic algorithms for solving Laplacian systems in subsequent sections.

\begin{figure}
\center
\includegraphics[width = \textwidth]{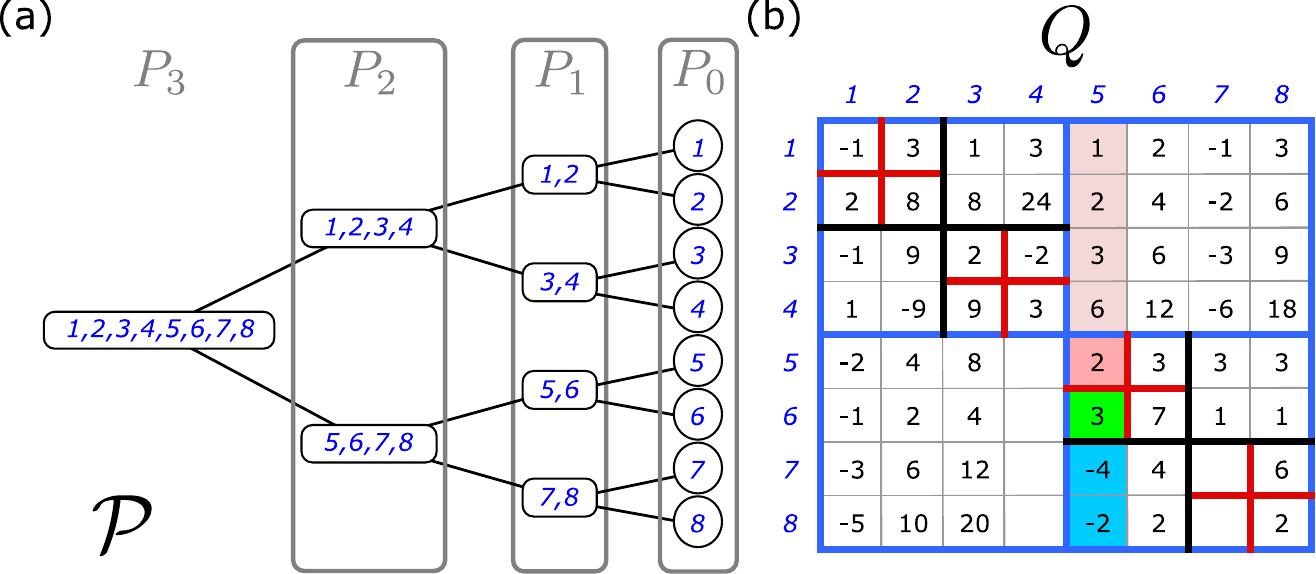}
\caption{(a) A dendrogram of $I = \{1, \ldots, 8\}$ of height $h=3$ and degree $d=2$. (b) The $\mathcal{P}$-partitioning of an $8 \times 8$ matrix where $\mathcal{P}$ is the dendrogram of $\{1, \ldots, 8\}$ in (a).}
\label{fig1}
\end{figure}

\begin{definition}[Dendrogram]
    A \emph{dendrogram} is a hierarchical partitioning $\mathcal P$ of the set $\{1,\ldots,n\}$.
    Every dendrogram comprises a sequence of increasingly finer partitions $P_h, \ldots, P_0$ starting from the coarsest (global) partition $P_h$ given by the whole set, up to the finest (singleton) partition $P_0$ into $n$ sets. 
    A dendrogram is conveniently represented by a rooted directed tree.
    The nodes of this tree at height $i$ are the subsets of partition $P_i$. 
    Thus the root ($i=h$) is the full set while the leaves ($i=0$) are the $n$ single-element subsets. 
    The children (out-neighbours) of a node at height $i$ correspond to the subsets of this node as specified by the next lower partition $P_{i-1}$.
    We call $h$ the \emph{height} of the dendrogram, and the maximum number of children of a node in the tree is denoted as \emph{maximum degree} $d$.
\end{definition}

Figure \ref{fig1}a shows an example of a dendrogram with height 3 and maximum degree 2.
For simplicity of notation and without loss of generality, we suppose throughout the paper that every node of a dendrogram has consecutive elements.

A dendrogram $\mathcal{P}$ induces a hierarchical block segmentation of a matrix $E\in \mathbb R^{n \times n}$ as follows. 
Let us denote the degree of the root node by $t\leq d$.
The rows and columns of $E$ are first block-partitioned according to the partition $P_{h-1}$:
\begin{equation}\label{Ematrix}
E = \left(\begin{array}{cccc} E_{I_1 \times I_1} & E_{I_1 \times I_2} & \hdots & E_{I_1 \times I_t} \\ E_{I_2 \times I_1} & E_{I_2 \times I_2} & \hdots & E_{I_2 \times I_t} \\ \vdots & \vdots & \vdots & \vdots \\ E_{I_t \times I_1} & E_{I_t \times I_2} & \hdots & E_{I_t \times I_t} \end{array}\right),
\end{equation}
where $I_1, \ldots, I_t$ are the elements of partition $P_{h-1}$. 
The diagonal blocks $E_{I_i \times I_i}$, are recursively sub-partitioned according to $P_{h-2}$, etc. 
This partitioning of $E$ is called \emph{$\mathcal{P}$-partitioning}.
See Figure~\ref{fig1}(b) for an illustration.

\begin{definition}(Elementary block)
    We use the term \emph{elementary block} to refer to a sub-matrix of $E$ generated by the $\mathcal{P}$-partitioning that is not further subdivided. 
    In other words it is a block of the form $E_{I_i \times I_j}$ where $I_i$ and $I_j$ are either two different sets in the same partition $P_k$, or two single-element sets of the finest partition $P_0$.
\end{definition}

\begin{definition}(Hierarchical Matrix)
    An $\mathcal{H}_r(\mathcal{P})$-matrix is a square matrix, structured according to the dendrogram $\mathcal{P}$, for which the elementary blocks have rank at most $r\in \mathbb N$. 
    We use the shorthand $\mathcal{H}_r$ when the dendrogram is clear from the context.

    Note that a sub-matrix $E_{I_i \times I_i}$ of an $\mathcal{H}_r(\mathcal{P})$-matrix $E$, where $I_i$ is a set of some partition $P_k$, is an $\mathcal{H}_r$-matrix as well.
\end{definition}

\subsubsection{Sparse factorization property}
In the following, we prove that $\mathcal{H}_r(\mathcal{P})$-matrices are $k$-sparsely factorizable, and express $k$ in terms of the rank $r$,  maximum degree $d$ and height $h$.

Recall that an $\mathcal{H}_r(\mathcal{P})$-matrix $E$ is of the form \eqref{Ematrix}.
Every non-elementary block $E_{I_i \times I_i}$ on the diagonal is recursively of the same form until the diagonal block is just a scalar.
Hence, every diagonal non-elementary block is a hierarchical matrix, too.
Further, note that every column of the full matrix $E$ is built by concatenating the corresponding columns of the $E_{I_i \times I_j}$ blocks. 
For example, the first column of $E$ can be built by stacking up the first columns of $E_{I_1 \times I_1}, E_{I_2 \times I_1}, \ldots, E_{I_t \times I_1}$. 

We can thus build a $k$-sparse factorization $E=CD$ as follows.
As every off-diagonal elementary block $E_{I_i \times I_j}$ has a rank of at most $r$, there is a matrix $D_{ij}$ such that the elementary block can be decomposed as $E_{I_i \times I_j} = C_{ij}D_{ij}$, where $C_{ij}$ has at most $r$ columns.
Thus, we know how to express all the elements in the off-diagonal blocks using this factorization.
Hence, if we knew a sparse decomposition of the diagonal blocks $E_{I_i \times I_i} = C_{ii} D_{ii}$, we could assemble the whole matrix $E$ by appropriate concatenation of the matrices $C_{ij}$.

To factorize the diagonal blocks we apply this construction recursively.
To make the recursion well defined, if the diagonal block $E$ is a scalar (a $1 \times 1$ matrix), we define $E = CD$, where $C$ is an arbitrary nonzero scalar, for instance we take $C = E$ and take $D=1$. 
Decomposing the columns of $E$ in this recursive way, we obtain a sparse factorization $E = CD$. 

We illustrate this for the case $t = 3$, hereafter.
For each $i \in \{1, 2, 3\}$, let each diagonal block $E_{I_i \times I_i} = C_{ii}D_{ii}$ be a $k_i$-sparse decomposition (recursively), and recall that each elementary block $E_{I_i \times I_j}$ $(i \neq j)$ can be factorized as $E_{I_i \times I_j} = C_{ij}D_{ij}$.
Then a $k$-sparse factorization of $E$ is given by:
\begin{flalign}    E 
    &=  \label{Cmatrix}
    \underbrace{\begin{pmatrix} 
            C_{11} & C_{12} & C_{13} & & & & & & \\
            & & & C_{22} &C_{21} & C_{23} & & & \\
            & & & & & & C_{33} & C_{31} & C_{32} 
        \end{pmatrix}}_{C}
    \underbrace{\begin{pmatrix}D_{11} & 0 & 0 \\ 
            0 & D_{12} & 0 \\ 
            0 & 0 & D_{13} \\ 
            \hline 0 & D_{22} & 0 \\ 
            D_{21} & 0 & 0 \\ 
            0 & 0 & D_{23} \\ 
            \hline 0 & 0 & D_{33} \\ 
            D_{31} & 0 & 0\\ 
            0 & D_{32} & 0
        \end{pmatrix}}_{D},
\end{flalign}
where $C_{11}, C_{22}, C_{33}$ are recursively defined according to the diagonal blocks of $E$.

Having thus found a possible factorization, the question remains what sparsity, $k$, it affords.
To answer this question, let us first consider the columns of $C$ necessary to build the first columns of $E$, and the union of their forward overlaps.
There are two types of columns in $C$ needed to build up the first block of columns in $E$.
\begin{enumerate}
    \item the columns in the $\begin{pmatrix}C_{11} & 0 & 0\end{pmatrix}^T$ block. 
    Their forward-overlap is $k_1+r(l-1)$, where $k_1$ is the sparsity of the factorization of $C_1$, and the $r(l-1)$ term accounts for the overlap with the $(l-1)$ $r$-column matrices $C_{12}$ and $C_{13}$. 
    \item The columns in the blocks $\begin{pmatrix}0 & C_{21} & 0\end{pmatrix}^T$ and $\begin{pmatrix}0 &0 & C_{31} \end{pmatrix}^T$. 
    Their forward overlap is $r(l-1)$ at most.
\end{enumerate}
As this argument holds for any column of $E$, the factorization $E=CD$ is $k$-sparse for  $k=\max_i k_i +r(l-1)$, where $k_i$ is determined recursively from the decomposition of the diagonal block $E_{I_i \times I_i}$.
Unravelling the recursion over all $h$ levels, we find that $k=rd(d-1)(h + 1)$, where $d$ is the maximal degree of the dendrogram, as before.

Throughout the paper, in a $k$-sparse factorization  $E = CD$ of an $\mathcal{H}_r(\mathcal{P})$-matrix, the matrix $C$ is supposed to be of the generic form \eqref{Cmatrix}, for an accordingly determined degree $d$. 
We will call this type of matrix a \textit{C-matrix}.
In the Appendix we prove that the number $p$ of columns of $C$ in the recursive construction in \eqref{Cmatrix} is bounded by $p \leq rd^2n$. 

We formalize the above findings in the following theorem.
\begin{theorem}
\label{thm41}
Let $E \in \mathbb{R}^{n \times n}$ be an $\mathcal{H}_r(\mathcal{P})$-matrix with a dendrogram $\mathcal{P}$ of height $h$ and maximum degree $d$. 
Then, there are matrices $C \in \mathbb{R}^{n \times p}$ and $D \in \mathbb{R}^{p \times n}$ such that $p \leq rd^2n$ and the factorization $E = CD$ is $k$-sparse for $k = rd(d-1)(h+1)$.
\end{theorem}

\subsection{Semiseparable matrices}
Another important matrix class which has received much attention in the literature are semi-separable matrices, whose inverses are given by tridiagonal matrices~\cite{Vandebril2008,Vandebril2007} and thus can be solved in linear time.

\begin{definition}\cite{Vandebril2005}
\label{defi600}
An $n \times n$ matrix $E$ is called $(p,q)$-semiseparable if the following relations are satisfied: $$\text{rank}(E(1: i+q-1, i:n)) \leq q \text{ and } \text{rank}(E(i:n, 1: i+p-1)) \leq p$$ for all feasible $1 \leq i \leq n$.
\end{definition}

\begin{theorem}\label{thm:semsep}
An $n \times n$ matrix that is $(p,q)$-semiseparable is an $\mathcal{H}_r(\mathcal{P})$-matrix where $r = \max\{p,q\}$ and $\mathcal{P}$ is a binary dendrogram.
\end{theorem}

The proof is given in the appendix.
It follows directly that semi-separable matrices are $k$-sparsely factorizable, too.
We note that, due to their remarkable structural properties, algorithms solving semiseparable systems in linear time are well known in the literature \cite{Vandebril2007, Vandebril2008}.

\subsection{Reduced incidence matrices of trees and their inverse}

In what follows, we define a reduced incidence matrix of a tree, and show that it is $k$-sparsely factorizable as it is an $\mathcal{H}_1(\mathcal{P})$-matrix where $\mathcal{P}$ is a binary dendrogram ($d = 2$). 
We remark that, to the best of our knowlege, this connection between hierarchical matrices and incidence matrices of trees has no been reported in the literature so far.
The importance of this observation arises in the context of Laplacian systems, as we will see in a later section.

We first give the definitions of an incidence matrix of a graph and of a reduced incidence matrix of a tree.

\begin{definition}[Incidence matrix, reduced incidence matrix]
\label{defi39}
Let $G$ be a positively weighted undirected graph on $n$ nodes and $m$ edges with an arbitrary direction chosen for each edge. An incidence matrix $B \in \mathbb{R}^{n \times m}$ of $G$ is a node-by-edge matrix such that given an edge $e_i$ of $G$ from node $i_1$ to node $i_2$ with weight $w_i$, the $i$th column of $B$ takes value $-\sqrt{w_i}$ at the source node $i_1$, value $\sqrt{w_i}$ at the target node $i_2$ and value $0$ at any other node.

A reduced incidence matrix of a graph $G$ is an incidence matrix of $G$ from which one row has been removed.
\end{definition}

To reveal the hierarchical structure in the reduced incidence matrix of a tree, one has to recursively split the nodes of the tree in a balanced way. 
A classic way to do so is provided by the tree-vertex-separator lemma.

\begin{lemma}[Tree Vertex Separator Lemma, \cite{Jordan1869,Chung1990}] 
\label{lem3}
For any forest $T$ with $n \geq 2$ nodes, one can  divide $T$ into two  forests both of at most $2n/3$ nodes, by removing at most one node $d$, which can be computed in $\mathcal{O}(n)$ time.
\end{lemma}

\begin{proposition}
\label{prop1}\label{prop47}
A reduced incidence matrix of an $n$-edge tree is, for some ordering of the nodes and edges, an upper-triangular $\mathcal{H}_1(\mathcal{P})$-matrix for a binary dendrogram $\mathcal{P}$  with height $\mathcal{O}(\log n)$. The inverse of the reduced incidence matrix is, for the same ordering of nodes and edges, also an upper-triangular $\mathcal{H}_1(\mathcal{P})$-matrix. The dendrogram $\mathcal{P}$ and both hierarchical matrices can be computed in time $\mathcal{O}(n \log n)$.  
Thus, a $\mathcal{O}(\log n)$-sparse factorization of (the inverse of) such a hierarchical matrix is computable in time $\mathcal{O}(n \log^2 n)$.
\end{proposition}

\begin{proof}
Note that in this proof we consider $T$ as an undirected tree with root $v$. 
A tree $T$ of $n$ nodes has $n-1$ edges, and hence is described by an $n$-by-$(n-1)$ incidence matrix. 
By convention we assign an arbitrary direction to each edge, encoded by the signs of the entries in the incidence matrix. 
However, the chosen direction does not play any role for the results in the following.
By removing a row from the incidence matrix, we obtain a square reduced incidence matrix of dimension $n-1$. 

We now split the tree $T$ into two forests $T_1$ and $T_2$ following the procedure of the Tree Vertex Separator Lemma.
Each of $T_1$, $T_2$ will accordingly have no more than $2n/3$ nodes. 
We assign the separator node $d$ (if any) to $T_2$. 
We now order the nodes in our reduced incidence matrix in two blocks according to this split:
\begin{equation*}
E = \begin{pmatrix}{E}_{I_1 \times I_1} & {E}_{I_1 \times I_2} \\ 0 & {E}_{I_2 \times I_2}\end{pmatrix},
\end{equation*}
where $E_{I_i \times I_i}$ (for $i=1,2$) is the reduced incidence matrix of $T_i$ and $E_{I_1 \times I_2}$ is a rank-$1$ matrix with at most one non-zero entry corresponding to the edge linking $d$ to its father.
Here, the indices of the edges have been assigned as follows: an edge connecting node $i$ and $j$ is indexed by $j$, if $j$ is one step further away from the root than $i$ (i.e. $j$ is the `child' of $i$).

We repeat this argument recursively and thereby create a dendrogram $P$ on the nodes of $T$ of height $\mathcal{O}(\log n)$, and a corresponding upper triangular $\mathcal{H}_1(\mathcal{P})$-matrix structure for $E$. 
From the ordering of edges, we see that the $i$th node is always incident to the $i$th edge, thus the diagonal entry of $E$ is $\pm \sqrt{w_i}$, making it easily invertible. 
Indeed, the inverse of $E$ can be computed recursively as 
$$E^{-1} = \begin{pmatrix} {E}_{I_1 \times I_1}^{-1} & F \\ 0 & {E}_{I_2 \times I_2}^{-1}\end{pmatrix} ,$$
with $F=-E_{I_1 \times I_1}^{-1}E_{I_1 \times I_2}^{} E_{I_2 \times I_2}^{-1}$.
Note that we may write $F= uv^T$ as it is clearly of rank one at most, thus leading to an upper-triangular $\mathcal{H}_1(\mathcal{P})$-matrix for $E^{-1}$ as well.
Both for $E$ and $E^{-1}$, every of the $\mathcal{O}(\log n)$ steps of the recursion takes $\mathcal{O}(n)$,  required to finding the tree vertex separators and (in case of $E^{-1}$) computing $u$ and $v$, solutions of triangular systems.
Therefore we get a total cost of $\mathcal{O}(n \log n)$.

Finally, using the procedure outlined above we can decompose $E^{-1}=CD$.
Using $E_{I_1 \times I_1}^{-1} = C_{11}D_{11}$ and $E_{I_2 \times I_2}^{-1} = C_{22}D_{22}$, we recursively construct:
\begin{equation*}
    E^{-1} = \begin{pmatrix} C_{11} & & u \\ & C_{22} & \end{pmatrix} \begin{pmatrix} D_{11} & \\ & D_{22} \\ & v^T \end{pmatrix}.
\end{equation*}
By unfolding this recursion we can see that this leads to a forward-overlap of size $\mathcal{O}(\log(n))$ in $C$, and an $\mathcal{O}(\log(n))$ column-sparse matrix $D$. Similarly, a $\mathcal{O}(\log n)$-sparse factorization can be obtained for $E$.
\end{proof}

\section{Fast iterative linear solvers on hierarchical systems}\label{sec:iterative_solvers}

To illustrate the usefulness of our results, in the following we showcase two concrete application scenarios in which the above developed theory can be employed.

\subsection{A strategy for solving underdetermined systems}
In the following, we focus again on the case of an underdetermined system $Ax=b$.
We devise a strategy that assumes a decomposition of the $n$-by-$m$ full-rank matrix $A$ (with $n < m$) of the form $A =\begin{pmatrix} E & F \end{pmatrix}$, where $E$ is an invertible $n \times n$ submatrix of $A$.
In particular, let us consider the case where $E^{-1}$ is hierarchical. 
We can then combine Theorem \ref{thm1} and the subsequent discussion, and Theorem~\ref{thm41} to obtain the following result.

\begin{theorem}
\label{thm32}
Let $A = \begin{pmatrix}E & F \end{pmatrix}$ be an $n \times m$ matrix with $n < m$, where $E \in \mathbb{R}^{n \times n}$ is invertible and $E^{-1}$ is an $\mathcal{H}_r(\mathcal{P})$-matrix with an associated dendrogram $\mathcal{P}$ of maximum degree $d$ and height $h$. Further, let $F \in \mathbb{R}^{n \times (m-n)}$ be $f$-column sparse. 
Then, we can compute an approximation of $x^* := \arg\min_{s.t. Ax = b} ||x||$ by applying $N$ iterations of the form (\ref{eq0}), in time 
$$\mathcal{O}(N  frd^2h + mf^2r^2d^4h^2) + Cost(CD),$$
where $Cost(CD)$ is the cost of computing a $(rd^2(h+1))$-sparse factorization of $E^{-1}$. The number of iterations to gain one order of magnitude on the error is at most $N_1 = \mathcal{O}(\|E^{-1}F\|_\text{Frob}^2+m)$. 
\end{theorem}
 
\begin{proof}
Following Theorem \ref{thm41}, let $E^{-1} = CD$ be a $k$-sparse factorization with  $k = rd(d-1)(h+1)=\mathcal O(rd^2h)$. 
By the second elementary property of the sparse factorization (see Property 2 on page \pageref{proper2}), we know that $C$ is $k$-row sparse and that each column of $D$ is $k$-sparse. 
A feasible solution to $Ax = b$ is then given by $x_0 = (\begin{smallmatrix} E^{-1}b \\ 0 \end{smallmatrix})$ where $E^{-1}b= CDb$ is computed in $\mathcal{O}(kn)$ time.

Now, consider the matrix $Q$ given in (\ref{eq10}). 
From our discussion above we know that the columns of $Q$ are a basis  of $\text{null}(A)$ and that the matrix $Q$ is $(kf+1)$-sparsely factorizable. 
Let $Q = \tilde{C}\tilde{D}$ be the $(kf+1)$-sparse factorization given in (\ref{eq11}).
We start from the vector $x_0$ and iteratively pick a column $q$ of $Q$ and perform an iteration of the form $(\ref{eq0})$. 
Theorem \ref{thm1} with $Q, \tilde{C}$ and $\tilde{D}$ then shows that the running time is given by 
$$\mathcal{O}(N  fk + mf^2k^2) + Cost(CD).$$ 
\end{proof}

\subsection{Square hierarchical systems}
The present technique can be also applied to solve square systems $Ax=b$, where $A$ is hierarchical and invertible. 

\begin{theorem}
The system $Ay=b$, where $A$ is an invertible $n$-by-$n$ $\mathcal{H}_r(\mathcal{P})$-matrix with $\mathcal{P}$ a dendrogram of degree $d$ and height $h$, can be solved iteratively in time $$\mathcal{O}(N rd^2h + nr^2d^4h^2) + Cost(CD),$$ where $N$ is the number of iterations and $Cost(CD)$ is the running time needed to compute a $k$-sparse factorization of $A$ with $k = rd(d-1)(h+1)$.
\end{theorem}

\begin{proof}
    In section \ref{overdetcase}, page \pageref{overdetcase}, we explain how to solve an overdetermined system using iterations (\ref{eq0}). 
    Trivially, we can use the presented method for the square system $Ay = b$. 
    Following the notations of Theorem \ref{thm1}, here $Q = A$, $m = n$ and the running time is 
    $$\mathcal{O}(Nk +  nk^2) + Cost(CD).$$ 
    We moreover use Theorem \ref{thm41} which states 
    $k = rd(d-1)(h+1)$ to deduce that the running time is 
    $$\mathcal{O}(Nrd^2h + nr^2d^4h^2) + Cost(CD).$$
\end{proof}

In particular, if $A$ is an $\mathcal{H}_1(\mathcal{P})$-matrix (rank $r=1$) with a binary ($d=2$) dendrogram $\mathcal{P}$ of height $h=\mathcal{O}(\log n)$ (e.g., $A$ could be the reduced incidence matrix of a tree), then this running time becomes 
$$\mathcal{O}(N\log n + n\log^2n),$$ 
where we have used Proposition \ref{prop1} which states that a sparse factorization of $A$ is computed in time $\mathcal{O}(n\log^2n)$.

As far as we know, this is the best iterative method in terms of cost per iteration ($\log n$). 
Most standard method would exhibit a cost of $\mathcal{O}(n)$ per iteration, the cost of a matrix-vector product.
However, for solving squared hierarchical systems a direct method exists that solves such a problem in $\mathcal{O}(n \log^2 n)$ \cite{ambikasaran}.

\section{Solving Laplacian systems in nearly linear time}
\label{section5}

In the following we demonstrate how the approach outlined above can be used to solve Laplacian systems.

\subsection{Minimum norm solution for a system with reduced incidence matrix}
\begin{corollary}
\label{cor51}
Let $A$ be a reduced incidence matrix of a connected undirected graph on $n$ nodes and $m$ edges. 
Then, the minimal norm solution $x^*$ of a compatible system $Ax = b$ can be computed with relative accuracy $\epsilon = \|x_t - x^*\| / \| x^*\|$ in $\mathcal{O}(m\log^2 (n) \log \log (n)  \log (m \epsilon^{-1}))$ time. 
\end{corollary}
\begin{proof} 
    Note that every edge in the graph corresponds to one column of $A$, and thus every spanning tree is associated with a submatrix $E$ which is invertible by construction~\cite{Strang1986}.
    Choosing an invertible (sub-)matrix $E$ such that $A =\begin{pmatrix}E & F\end{pmatrix}$ is therefore equivalent to selecting a spanning tree of $G$.
    We now claim that we can choose $E$, i.e., choose an appropriate spanning tree, such that $\|E^{-1}F\|_\textrm{Frob}^2=\mathcal{O}(m \log n \log \log n)$.

    For any choice of spanning tree, we define the root as the node whose row has been removed from the incidence matrix $A$ to obtain a reduced incidence matrix. 
    We choose the (arbitrary) orientation on the edges so as to go from root  to leaves. 
    We also order the nodes from root to leaves (topological order) and edges so that any edge has the same index as its destination.  
    Let us call the \textit{unweigted, directed} adjacency matrix of this spanning tree $T_E$.
    With the choices made above $T_E$ is upper triangular. 
    Then we can write $E=(I-T_E)\sqrt{W_{T_E}}$ where $W_{T_E}$ is the diagonal matrix weights on the edges. 

    Using a Neumann series expansion we can see that $E^{-1}=W^{-1/2}_{T_E} (I+T_E+T_E^2 + T_E^3 + \ldots + T_E^h)$ where $h$ is the height of the tree. 
    The columns of $E^{-1}$ encode the paths between root and leaves, with entries given by the (positive) inverse square root of the edge-weights.

    Since $F$ is a (reduced) incidence matrix, each column $i$ of $E^{-1}F$ is the (weighted) difference between two columns of $E^{-1}$.
    In fact, each column $i$ of $E^{-1}F$ describes the (signed) path in the tree between the extremities of edge $i$, on which each edge $e$ has weight $\sqrt{w_i/w_e}$.
    Therefore the squared Frobenius norm of $E^{-1}F$ is the so-called stretch of the tree in the graph with \emph{inverse} weights, i.e. weight $w^{-1}_e$ on each edge $e$ of the graph, as already noticed in \cite{Kelner2013}. Using the algorithm in Ref. \cite{Abraham2012} we can therefore find a spanning tree with reduced incidence matrix $E$ such that $\|E^{-1}F\|^2_\textrm{Frob}= \mathcal{O}(m \log n \log \log n)$, where $m$ is the number of edges in the graph.
    The incurred computational cost for is $\mathcal{O}(m \log n \log \log n)$~\cite{Abraham2012}.

    From Proposition \ref{prop47}, it follows that $E^{-1}$ is an $\mathcal{H}_1(\mathcal{P})$-matrix, with is a binary dendrogram $\mathcal P$ of height $h = \mathcal{O}(\log n)$, and parameters $r=1$, $d=2$.
     A sparse decomposition of $E^{-1}$ can thus be computed in time $\mathcal{O}(n\log^2n)$. 
     Using Theorem \ref{thm32}, we can thus compute the minimal norm solution $x^*$ of $Ax = b$ in nearly linear time. 
     
     More precisely, following Section \ref{sec:underdetermined_convergence} we define  $\kappa=\|E^{-1}F\|_\text{Frob}^2+m$,  which is $\mathcal O (m \log n \log \log n)$. We then find that $N_1=\mathcal{O}(\kappa)$ iterations, each of which costs $k=\mathcal O (\log n)$, suffice to gain one order of magnitude, and the overall cost to achieve a relative accuracy $\epsilon$ is $\mathcal O (\kappa \log(\kappa \epsilon^{-1}) + m k^2)$, which in this case reduces to $\mathcal O (m \log^2 n \log \log n \log (m\epsilon^{-1}))$.
     
\end{proof}

\subsection{Solving Laplacian systems}

The above corollary provides the critical step in solving a compatible Laplacian system $L\chi=c$, where $L$ is the Laplacian of the same graph, as we show now. 
For a given incidence matrix $B$ the Laplacian is defined as $L=BB^T$ , or equivalently as the node-by-node matrix with entries $L_{ij}=-w_{ij}$ for every edge $ij$ of weight $w_{ij}$, $L_{ij}=0$ if $i$ is not adjacent to $j$, and the weighted degree $L_{ii}=\sum_k w_{ik}$ on diagonal entries. 
Such a system $L\chi=c$ can be solved in two steps:
\begin{enumerate}
\item solve $Bx=c$ so that $x$ is in the image of $B^T$; 
\item  solve the compatible, overdetermined system $B^T\chi=x$.
\end{enumerate}
This strategy of splitting the problem of solving a Laplacian system into 2 parts is in line with the approach followed by Kelner et al.~\cite{Kelner2013}.
However, their algorithm relies on graph-theoretic notions and a specific data structure construction, rather than a matrix decomposition.

Note that the first step in the procedure above is equivalent to finding the minimum-norm solution of $Bx=c$.
Any solution of $Bx=c$ is of the form $x=B^T\chi+v$, for some $v$ such that $Bv=0$. 
This implies that $v$ is orthogonal to $B^T\chi$, and thus $B^T\chi+v$ has a norm larger than $B^T\chi$, with the minimum norm solution given by $v=0$. 
The goal is therefore to solve $Bx=c$ in the minimum norm sense. 
Since the columns of $B$ sum to zero, we can remove an arbitrary row without affecting the solution, i.e., we can `ground' the system.
Let us call $A$ the so-obtained reduced incidence matrix of the graph, and $b$ the vector obtained from $c$ by removing one entry.
Now we have to solve $Ax=b$, which can be done efficiently as discussed above.

The second step outlined above then requires finding the solution of a compatible overdetermined system.
This can be found by solving the square invertible triangular subsystem $E^Ty=x_E$  where $E$ is the reduced incidence matrix of the spanning tree used to solve $Ax = b$ (see the proof of Corollary \ref{cor51}) and $x_E$ is the corresponding part of vector $x$. 
Solving this triangular system takes $\mathcal{O}(n)$ time, from leaves to root. 

We remark that when solving a semi-definite positive system $L\chi=c$, the $L$-pseudo-norm $\|\chi\|_L^2=\chi^TL\chi$ is often used as the error norm. 
Note that all $\|\chi\|_L^2$ vanishes only if vector $\chi$ has identical entries. 
The relative accuracy of the solution $\chi$ is accordingly defined as $\epsilon = \|\chi-\chi^*\|_L/\|\chi^*\|_L$. 

Putting these pieces together, we obtain the following theorem: 
\begin{theorem}
\label{thmLaplacian}
Given a Laplacian matrix $L$ of a connected graph with $m$ edges and a zero-sum vector $c$, the (compatible) system $L\chi=c$ can be solved within time $\mathcal{O}(m \log^2 n \log \log n \log (m \epsilon^{-1}))$ with relative accuracy $\epsilon$, as measured in the $L$-pseudo-norm. 
\end{theorem}

 \begin{proof} 
 From Corollary \ref{cor51} we find an approximate solution $x^* +\Delta x$ to the
 problem $Bx=c$, with $\|\Delta x\|/\|x^*\| \leq \delta$, in time 
 $\mathcal{O}(m \log^2 n \log \log n  \log (n \delta^{-1}))$. 
 
 We then find the approximate solution $\chi^*+\Delta \chi$ as $E^{-T}(x_E^*+\Delta x_E)$, where $x_E$ denotes the  restriction of the $m$-dimensional vector $x$ to the $n$ entries corresponding to $E$. 
 The incurred error $\Delta \chi$ can be bounded, using $L=BB^T$ and $B=(E \,\,\, F)$: 
 \begin{flalign}
\|\Delta \chi\|^2_L&=\|E^{-T} \Delta x_E\|^2_L =\|(I \,\,\,\,\, E^{-1}F)^T \Delta x_E\|^2 \leq \mathcal{O}(m \log n \log \log n) \|\Delta x\|^2  
\end{flalign} 
 
Moreover the exact solution fulfills  $\|\chi^*\|^2_L=\|x^*\|^2$ by definition of $x=B^T\chi$. 
Thus, we see that the relative accuracy on $x$ in terms of $\|.\|_L$ is 
$$
\frac{\|\Delta \chi\|^2_L}{\|\chi^*\|^2_L}=\mathcal{O}(m\log n \log \log n )\frac{\|\Delta x\|^2}{\|x^*\|^2}
$$
Therefore we can choose $\delta^{-1}=\mathcal{O}(\sqrt{m\log n \log \log n} ) \epsilon^{-1}$, for any required accuracy level $\epsilon$.
The proof is concluded by Corollary \ref{cor51}.
\end{proof}

We remark that the computational complexity of our final algorithm could be reduced further, by using some of the computational techniques discussed in \cite{Koutis2010,Koutis2011,Kelner2013}, which are beyond the scope of this paper, however.
For instance, one could employ a preconditioning to change the norm of $\|E^-{1}F\|$ and thereby obtain a better initial estimate for $\hat x_0$.
Indeed using such a preconditioning recursively, Kelner et al. are able to obtain an algorithm with a total complexity of $\mathcal{O} (m \log^2 n \log \log n \log \epsilon^{-1})$~\cite{Kelner2013}.
Note, however, that Kelner et al.~\cite{Kelner2013} employ quite different means to establish this result. 
Instead of a matrix factorization, the core tool invoked is an efficient data-structure which enables fast updates.
Our $k$-sparse matrix factorization approach may thus be seen as an alternative perspective on the problem of solving Laplacian systems.

\section{Conclusion}
In this paper we have considered the problem of finding the minimum norm vector $x$ within an affine space, which arises naturally when solving an under- or overdetermined linear system. 
We have shown that this problem can be solved very efficiently in an iterative manner by choosing the matrix of search directions $Q= [q_1,\ldots,q_m]$ in an appropriate way.
Specifically, if there exists a $k$-sparse matrix factorization of $Q$, each iterative update of the form $x_{t+1} = x_t - \frac{x_t^Tq_i}{q_i^Tq_i}  q_i$ can be computed in $\mathcal{O}(k)$ time, enabling us to construct fast algorithm for solving linear systems.
The notion of a $k$-sparse matrix factorization is indeed central to these findings, as it ensures the existence of a computationally efficient update scheme despite the fact that $Q$ might be full, i.e., the search directions are not formed by sparse vectors.

We have shown that some important classes of matrices are $k$-sparsely factorizable, and in particular that in the case of hierarchical matrices $k$ does not depend on the size of the matrix, but merely on the depth of the hierarchy. 
From this, we have deduced an iterative method with fast iterations that approximates the minimal norm solution of underdetermined linear systems. 
In particular, this approach can be applied when the coefficient matrix is the incidence matrix of a connected graph.
This leads naturally to a method to solve Laplacian systems in nearly-linear time. 
In this context, our work provides a complementary algebraic perspective to the problem of solving Laplacian system, and connects combinatorial and graph-theoretic notions with the problem of finding a $k$-sparse matrix factorization.

An important direction for future work is to characterise the general class of matrices that can be sparsely factorized in more detail, and see how it can be extended beyond the matrices discussed within the present manuscript.
For instance, solvers based on tensor decompositions~\cite{Kressner2010,Ballani2013,Oseledets2012} have been presented in the literature, which assume that the linear system under study has an inherent Kronecker-product~\cite{Kressner2010,Ballani2013} or tensor-train~\cite{Oseledets2012} representation (or at least can be well approximated by such a structure).
It would be interesting to investigate in how far these matrix structures are also amenable to a $k$-sparse factorization.
    
Other avenues for future work include investigating possible parallelization of the here presented techniques, or combining them with other randomized update schemes ~\cite{Gower2015,Gower2015a} than the here considered randomized Kacmarz updates~\cite{Strohmer2009}.
For instance, it would be interesting to see in how far block updates (instead of single coordinate updates), could lead to more efficient iterative algorithms.

\bibliographystyle{siamplain}
\bibliography{references}

\newpage
\appendix

\section{Proof of Theorem \ref{thm1}}
\begin{theorem}[Theorem \ref{thm1}]
Let $Q \in \mathbb{R}^{m \times n}, C \in \mathbb{R}^{m \times p}$ and $D \in \mathbb{R}^{p \times n}$ be matrices such that $Q = CD$ is a $k$-sparse factorization of $Q$, and consider iterations of the form (\ref{eq0}) that start from an arbitrary vector $x_0 \in \mathbb{R}^m$. 
If every $q_i$ in \eqref{eq0} is a column of $Q$, then the computational complexity of running $N$ iterations of \eqref{eq0} is: 
$$\mathcal{O}(Nk + (m+n)k^2).$$
With the same complexity, we can compute a $y_N$ such that $x_N = x_0 + Qy_N$, where $x_N$ denotes the vector resulting from the $N$ first iterations. 
\end{theorem}

\begin{proof} 

    Let us first comment on the general strategy for computing fast iterations.
    Given $x_t \in \mathbb{R}^m$ and a column $q_j = Cd_j$ of $Q$, recall that the next iteration we aim to compute is of the form 
    \begin{equation}\label{eq:iteration}
        x_{t+1} = x_t - \frac{x_t^Tq_j}{q_j^Tq_j}  q_j.
    \end{equation}

    In order to get a running time for each iteration not depending on the system size $m$, we make use of two generating sets of $\mathbb{R}^m$.
    The sets are given by the columns of $C$, as well as the columns of $CU^{-T}$, where $U$ is the $p \times p$ upper triangular matrix such that $C^TC = U^T + U$.
    Each column of $Q$ has a decomposition in terms of these generating sets with a sparsity governed by $k$; indeed a column $q_j$ is expressed as $q_j=C d_j = CU^{-T}U^Td_j$, where $d_j$, a column of $D$, is $k$-sparse and $e_j := U^Td_j$ is $k$-sparse by definition of the $k$-sparse factorisation.
    Using these sets we can thus express $x_t$, with the coefficient vectors $h_t, g_t$, defined via the relationships $x_t=Ch_t$ and $x_t=CU^{-T}g_t$. 
    Note that such a vector $g_t$ is given by $g_t=U^Th_t$.
    Now at each iteration, we only use the vectors $h_i, g_i, d_j$ and $e_j$, and do not need to store the full vector $x_t$. 
    In particular the inner-product can be computed as:
    \begin{align*}
    x_t^Tq_j &= h_t^T(C^TC)d_j= h_t^T(U + U^T)d_j\\
    &= (U^Th_t)^Td_j + h_t^T(U^Td_j)= g_t^Td_j + h_t^Te_j.
    \end{align*}
    This can be done in $\mathcal{O}(k)$ time, as we will show in the following.
    
    In order to establish this key result about the complexity of the inner product, which leads directly to an efficient algorithm for performing our iterative updates, we will proof the following facts.
    \begin{enumerate}
        \item[\textbf{Fact 1}] We can compute the matrix $U$ in $\mathcal{O}(mk^2)$ (which is also the cost of computing $C^TC$)
        \item[\textbf{Fact 2}] We can compute an $m$-sparse vector $h_0 \in \mathbb{R}^p$ such that $x_0 = Ch_0$ in time $\mathcal{O}(m)$
        \item[\textbf{Fact 3}] We can compute $g_0 := U^Th_0$ in time $\mathcal{O}(mk)$.
        \item[\textbf{Fact 4}] The matrix $U^TD$ can be computed in time $\mathcal{O}(nk^2)$ 
        \item[\textbf{Fact 5}] All the scalar products $q_i^Tq_i$, where $q_i$ is a column of $Q$ are computable in time $\mathcal{O}(nk)$
    \end{enumerate}

    \textit{\textbf{Proof of Fact 1}}
    The cost of computing $C^TC$ can be estimated by the number of scalar additions and multiplications involved. 
    In fact the number of additions is the same as the number of multiplications, so we need only track the number of scalar multiplications. 
    From the elementary properties of the $k$-sparse factorization, it follows that there are at most $k$ entries per row. 
    In the  course of computing the entries of $C^TC$ all the scalar products between the $p$ columns of $C$ will be computed.
    Thus we find that every entry of the first row of $C$ will be multiplied with every of the $k$ (or less) entries of first row, which gives $\mathcal{O}(k^2/2)$ scalar multiplications associated to the entries of the first row. 
    Since every row can be treated similarly, the cost of computing $C^TC$ is at most $\mathcal{O}(mk^2)$.

    \textit{\textbf{Proof of Fact 2}}
    We can assume without loss of generality that the columns of $C$ contain the canonical basis of $\mathbb{R}^m$. 
    To see this, one can set ${\tilde{C} := \begin{pmatrix} I_m & C\end{pmatrix} \in \mathbb{R}^{m \times (p+m)}}$ and ${\tilde{D} := \begin{pmatrix} 0 & D^T \end{pmatrix}^T \in \mathbb{R}^{(p+m) \times n}}$. 
    It then follows that for each column $\tilde{c}_i$ of $\tilde{C}$, $|FO(\tilde{c}_i)| \leq k+1$, that $\tilde{D}$ is $(k+1)$-column sparse and that for each column $\tilde{d}_j$, $\left|\cup_{i \in supp(\tilde{d}_j)} FO(\tilde{c}_i)\right| \leq k+1$. 
    Consequently, even though $\tilde{D}$ has some zero rows, the factorization $\tilde{C}\tilde{D}$ has all the properties of a $(k+1)$-sparse factorization and we say that $\tilde{C}\tilde{D}$ is $(k+1)$-sparse. 
    As a consequence, the running time does asymptotically not depend on the choice of the decomposition  $CD$ or $\tilde{C}\tilde{D}$.
    Hence, we can assume without loss of generality that a vector $h_0 \in \mathbb{R}^p$, such that $x_0 = Ch_0$, can be computed in $\mathcal{O}(m)$ time.
    
    \textit{\textbf{Proof of Fact 3}}
    Denote by $U$ the $p \times p$ upper triangular matrix such that $C^TC = U^T + U$. 
    Notice that the $i^{th}$ row of $U$ is $|FO(c_i)|$-sparse. 
    Since $|FO(c_i)| \leq k$, this implies that the matrix $U$ is $k$-row sparse. 
    Moreover, as each column of $C$ is assumed to be nonzero, we can deduce that $U$ is invertible.
    Hence, given $h_0$, since $U^T$ is $k$-column sparse, we compute the vector $g_0 := U^Th_0$ in time $\mathcal{O}(mk)$.

    \textit{\textbf{Proof of Fact 4}}
    Let $d_j$ be a column of $D$, which is $k$-sparse. Then, since $Q = CD$ is a $k$-sparse factorization, the vector $e_j := U^Td_j$ is $k$-sparse and is computed in time $\mathcal{O}(k^2)$.
    Consequently, we can compute the matrix product $U^TD$, i.e., all vectors $e_i$ in $\mathcal{O}(nk^2)$.

    \textit{\textbf{Proof of Fact 5}} We compute any product $q_i^Tq_i$ as:
    \begin{align*}
        q_i^Tq_i & = d_i^T(C^TC)d_i= d_i^T(U + U^T)d_i = (U^Td_i)^Td_i + d_i^T(U^Td_i) = e_i^Td_i + d_i^Te_i.
    \end{align*}
    Since $e_i$ and $d_i$ are $k$-sparse, it takes $\mathcal{O}(k)$ time to compute $q_i^Tq_i$, and thus $\mathcal{O}(nk)$ to compute all the products.

    \section{Fast iterative algorithms}
    Following the analogous reasoning as in the proof of fact 5, we see that
    \begin{align*}
    x_i^Tq_j = g_i^Td_j + h_i^Te_j
    \end{align*}
    is also computable in $\mathcal{O}(k)$ time.
    Hence, we can compute a first iteration of \eqref{eq:iteration} efficiently.
    
    In order to make this computational gain available at every iteration, we have to find a way to update $h_t$ and $g_t$ in a fast manner, too.
    Given $h_t, g_t \in \mathbb{R}^p$ such that $x_t = Ch_t$ and $g_t = U^Th_t$ and given $e_j = U^Td_j$, the vectors 
    \begin{equation*}
        h_{t+1} := h_{t} - \frac{x_t^Tq_j}{q_j^Tq_j}  d_j 
    \end{equation*}
    \begin{equation*}
        g_{t+1} = g_t - \frac{x_t^Tq_j}{q_j^Tq_j}  e_j 
    \end{equation*}
    are such that $x_{t+1} = Ch_{t+1}$ and $g_{t+1} = U^Th_{t+1}$. 
    Moreover, from fact 2 and 3, and the sparsity of $d_j$, it follows that the vectors $h_{t+1}$ and $g_{t+1}$ are computed in time $\mathcal{O}(k)$.

    Consequently, at each iteration, we only need the vectors $h_t, g_t, d_j$ and $e_j$ in order to compute $h_{t+1}$ and $g_{t+1}$.
    Note that both $h_{t+1}$ and $g_{t+1}$ are required to compute the scalar product $x_{t+1}^Tq_{j}$ (needed in the next iteration)
    in time $\mathcal{O}(k)$. 
    Finally, the approximate solution after $N$ steps, $x_N$, is computed from the relation $x_N = Ch_N$.
    This can be done in $\mathcal{O}(mk)$ time due to the sparsity of $C$.

    Combining these results, it follows that $N$ iterative updates can be performed in time 
    $$\mathcal{O}(Nk + pk + (m+n)k^2).$$ 

    Finally, the computation of $y_N$ such that $x_N = x_0 + Qy_N$ can be performed while computing $x_N$ with the above described method without additional costs. 
    Indeed, start with $y_0 = 0$. If the $(t+1)^{th}$ iteration is 
    $$x_{t+1} = x_t - \frac{x_t^Tq_j}{q_j^Tq_j}q_j,$$ 
    then $y_{t+1}$ corresponds to updating the $j^{th}$ entry of $y_t$ by adding $-\frac{x_t^Tq_j}{q_j^Tq_j}$.
    As the required scalar products are computed for $x_{t+1}$, no additional cost is incurred.
\end{proof}

\subsection{Relationships to randomized Kaczmarz and randomized coordinate descent}
In the following we discuss how the iterative updates we discuss in Section \ref{sec:prelim} can be interpreted from the lens of (randomized) Kacmarz and (randomized) coordinate descent methods.

\subsubsection{The underdetermined case}
We consider finding the minimum norm solution for a consistent linear system $Ax = B$ where $A$ is an $n \times m$ matrix with $m > n$.
As discussed in Section \ref{sec:prelim}, given any initial solution $x_0$, this can be achieved by iteratively updating $x$, by projecting it onto the hyperplane orthogonal to the vectors $q_i$:
\begin{equation} \label{eq:update}
    x_{t+1} = \left[I- \dfrac{q_iq_i^T}{q_i^Tq_i}\right ] x_t = x_t - \dfrac{x_t^Tq_i}{\|q_i\|^2} q_i,
\end{equation}
where the update directions $q_i$ lie within the null-space of $A$.
Stated differently, the matrix $Q = [q_1,q_2,\ldots]$ fulfils $AQ=0$.

Now we can relate the above to the Kacmarz scheme as follows:
Let us denote the row vectors of $A$ by $a_i^T$ ($i\in 1,\ldots n$).
One update step according to the Kacmarz scheme is defined as:
\begin{equation} \label{eq:KM}
    x_{t+1} = x_t + \dfrac{b^i - a_i^Tx_t}{a_i^Ta_i^{}}a_i, 
\end{equation}
where $b^i$ is the $i$-th component of the right hand side.

To see that finding this minimal norm solution via the update $\eqref{eq:update}$ can indeed be interpreted as Kacmarz update scheme, let us define the augmented $m\times m$ linear system:
\begin{equation}\label{eq:augmented_sys}
A'x = \begin{pmatrix} A \\ Q^T\end{pmatrix}x = \begin{pmatrix} b\\ 0 \end{pmatrix}.
\end{equation}
Note the (unique) solution to this system is indeed the minimum norm solution of $Ax=b$.

Let us now consider iteratively solving~\eqref{eq:augmented_sys} according to the Kacmarz scheme.
Since we assumed that we start with an initial condition $x_0$ that fulfills $Ax_0 =b$, the first $m$ equations are automatically fulfilled.
Given that the right hand side has to be zero for the $m-n$ equations for the solution to be of minimum norm, we can easily see that all the updates are indeed of the desired form.

\paragraph{Finding a feasible solution $x_0$}
Let us briefly discuss the scenario that we cannot obtain a feasible solution $x_0$ in a simple manner, but the matrix $(A')^T$ in~\eqref{eq:augmented_sys} is sparsely factorizable.
As using our $k$-sparse factorization, all inner-products are cheap to compute, we can also compute iterations of the form~\eqref{eq:KM} efficiently.
In particular, for a compatible square system of the form $Ax = b$, where $A^T$ is sparsely factorizable (say $Q = I A^T $ is $k$-column sparse), we can employ our $k$-sparse matrix factorization to compute any iteration of the form~\eqref{eq:KM} in $\mathcal O(k)$ time.

\subsubsection{The overdetermined case}
In this case we have a system of the form $Ay  = b$ where a is an $n \times m$ matrix with $m < n$.
Let us define $x = Ay-b$ as discussed in Section \ref{sec:prelim}.
From the analytical solution to the normal equations $A^TAy = A^Tb$ we know that we must have $A^Tx =0$.
Whence, if we choose $Q=A$ in our update rule $\eqref{eq:update}$, this is exactly equivalent to an update of the form $\eqref{eq:KM}$, and can be solved efficiently using our $k$-sparse matrix factorization.

As discussed by Gower and Richtarik \cite{Gower2015, Gower2015a} the dual update in $y$ simply corresponds to coordinate descent:
$$ y_{t+1} = y_t - \dfrac{(Qy-b)^Tq_i}{\|q_i\|^2}e_i,$$
where $e_i$ is the $i$th unit vector.
Indeed by keeping track of the step sizes $\alpha_t^* = (q_i^Tx_t / q_i^Tq_i) q_i$ we effectively construct $y^*$ such that $Qy^* + x_0 = x^*$ in \eqref{eq:aff}.

\section{Semiseparable and hierarchical matrices}

\begin{lemma}
\label{lem43}
The number of columns in $C$ in the recursive construction in \eqref{Cmatrix} in Section \ref{sec:31} is given by $p \leq rd^2n$.
\end{lemma}
\begin{proof} By induction on $n$, we prove that $p \leq rd(d-1)\left(\frac{d}{d-1}n - \frac{1}{d-1}\right)$.
\begin{enumerate}
\item If $n = 2$, then $d = 2$ and $p \leq 4  \leq  rd(d-1)\left(\frac{d}{d-1}n - \frac{1}{d-1}\right) \leq rd^2n$. 
\item If $n > 2$, then $E$ is of the form (\ref{Ematrix}). 
    Let us denote the size of a diagonal block $E_{I_i \times I_i}$ by $n_i$, so we have $n = \sum_{i = 1}^d n_i$.
    Now, from the  construction of $C$ we know that $p \leq r(d-1)d + \sum_{i = 1}^d p_i$, where $p_i$ is the maximum number of columns in the matrix $C_i$ of $E_{I_i \times I_i}$ $(1 \leq i \leq d)$. 
    Consequently, by induction we have 
\begin{align*}
p &\leq r(d-1)d + r(d-1)d \sum_{i = 1}^d \left(\frac{d}{d-1}n_i - \frac{1}{d-1}\right)\\  
    &= r(d-1)d\left(\frac{d}{d-1}n - \frac{1}{d-1}\right) \leq rd^2n.
\end{align*}
\end{enumerate}
\end{proof}

\begin{theorem}[Theorem \ref{thm:semsep}]
An $n \times n$ matrix that is $(p,q)$-semiseparable is an $\mathcal{H}_r(\mathcal{P})$-matrix where $r = \max\{p,q\}$ and $\mathcal{P}$ is a binary dendrogram.
\end{theorem}
\begin{proof} Following the definition, we have $1 \leq p,q \leq n$ and we assume without loss of generality that $n \geq 2$.
Now, let $E$ be an $n \times n$ matrix which is $(p,q)$-semiseparable and let $\{I_1, I_2\}$ be a partition of $\mathcal I =\{1, ..., n\}$ with $I_1 = \{1, ..., \lfloor\frac{n}{2}\rfloor\}$ and $I_2 = \mathcal I \backslash I_1$.
Consider an integer $i_1 \in \mathcal I$ such that $ \lfloor\frac{n}{2}\rfloor - q +1 \leq i_1 \leq \min\{\lfloor \frac{n}{2}\rfloor, n-q +1\}$. 
Then, the submatrix ${E(1: i_1+q-1, i_1:n)}$ is of rank $\ell_1 \leq q$ and contains $E_{I_1 \times I_2}$.

Similarly, if $i_2 \in \mathcal I$ such that $ \lfloor\frac{n}{2}\rfloor - p +1 \leq i_2 \leq \min\{\lfloor \frac{n}{2}\rfloor, n-p +1\}$, then the submatrix ${E(i_2: n, 1:i_2+p-1)}$ is of rank $\ell_2\leq p$ and contains $E_{I_2 \times I_1}$.
Therefore, we have shown that the off-diagonal blocks $E_{I_1 \times I_2}$ and $E_{I_2 \times I_1}$ are of a rank smaller of equal to $r$.

From the definition of semiseparable matrix, it follows that the diagonal blocks $E_{I_1 \times I_1}$ and $E_{I_2 \times I_2}$ of $E$ are also $(p,q)$-semiseparable matrices.
Repeating the previous argument recursively on $E_{I_1 \times I_1}$ and $E_{I_2 \times I_2}$ shows that $E$ is an $\mathcal{H}_r(\mathcal{P})$-matrix with $\mathcal{P}$ being a binary dendrogram, i.e $d = 2$.
\end{proof}

\section{Convergence rate and required number of iterations for randomly sampled search directions}

The proof is due to Strohmer and Vershynin \cite{Strohmer2009} and has originally been given in the context of a randomized Kaczmarz's method for solving linear systems.
The version we give here is adapted to the context of this paper. 

We want to establish the speed of convergence of iterations \eqref{eq0}, when each column $q_i$ of the matrix $Q$ is chosen with probability proportional to $\|q_i\|^2$. 
In order to do so, for any $x$  we first consider the auxiliary quantity $$\sum_{i} \langle x,q_i \rangle^2 = x^TQQ^Tx \geq \sigma^2_{\min}(Q) \|x\|^2.$$
Here $\langle  x,q_i \rangle$ denotes the usual scalar product $x^Tq_i$. 
If each direction $q_i$ is selected with probability  $p_i=\|q_i\|^2/\sum_j \|q_j\|^2=\langle q_i,q_i \rangle/\|Q\|^2_\textrm{Frob}$, we can rewrite this inequality as $$\sum_{i} p_i \langle x,q_i \rangle^2/\langle q_i,q_i \rangle \geq \frac{\sigma^2_{\min}(Q)}{\|Q\|^2_\textrm{Frob}} \|x\|^2.$$

Now, we know that $x^*$, the minimum norm point in the affine space $x_0+ \textrm{span}\{q_i\}$, must be orthogonal to all directions $q_i$ in the space.
It thus follows that $\langle x, q_i \rangle=\langle x-x^*, q_i \rangle$.
Therefore, we can write:
\begin{equation} \label{eqStrohmer}
\sum_{i} p_i \langle x,q_i \rangle^2/\langle q_i,q_i \rangle \geq \frac{\sigma^2_{\min}(Q)}{\|Q\|^2_\textrm{Frob}} \|x-x^*\|^2 .
\end{equation}
  
Furthermore, we have
$$\|x_{t}-x^*\|^2=\|(x_{t+1}-x^*)-(x_{t+1}-x_{t})\|^2=\|x_{t+1}-x^*\|^2+\|x_{t+1}-x_{t}\|^2.$$
The second equality is due to orthogonality  
$$\langle x_{t+1}-x^*, x_{t+1}-x_t \rangle=0=\langle x_{t+1}-x^*, \text {const}\cdot q_i \rangle . $$
This can be checked from the two following observations. 
First, $x^*$ is orthogonal to all directions $q_i$ in the affine space, as it is the point with the minimal norm in our affine space.
Second, $x_{t+1}$ is computed as the minimum norm point on the line $x_t+\alpha_t q_i$, and is therefore also orthogonal to the current search direction $q_i$. 
Thus the error $x_{t+1}-x^*$ is also orthogonal to search direction $q_i$.

Finally we combine the results and observe that the expected value of the error norm

\begin{flalign}
\sum_i p_i& \|x_{t+1}-x^*\|^2= \sum_i p_i \|x_{t}-x^*\|^2 - \sum_i p_i \|x_{t+1}-x_{t}\|^2\\
=& \|x_{t}-x^*\|^2 - \sum_i p_i \frac{\langle x_t, q_i \rangle^2}{\langle q_i,q_i \rangle^2}\|q_i\|^2
\leq  \left (1 - \frac{\sigma^2_{\min}(Q)}{\|Q\|^2_\textrm{Frob}}\right ) \|x_{t}-x^*\|^2, 
\end{flalign}
which is  the desired result.

\end{document}